\documentclass{article}
\usepackage{amssymb}
\usepackage{amsmath}
\usepackage{amsfonts}
\usepackage{geometry}
\usepackage{graphicx}

\newtheorem{theorem}{Theorem}[section]

\newtheorem{corollary}[theorem]{Corollary}

\newtheorem{lemma}[theorem]{Lemma}

\newtheorem{proposition}[theorem]{Proposition}
\newtheorem{remark}[theorem]{Remark}

\newenvironment{proof}[1][Proof]{\noindent\textbf{#1.} }{\ \rule{0.5em}{0.5em}}
\makeatletter
\def\@makecaption#1#2{\vskip\abovecaptionskip
\hb@xt@\hsize{\hfil#2\hfil}\vskip\belowcaptionskip}
\makeatother

\begin{document}

\title{Catalan States of Lattice Crossing}
\author{Mieczyslaw K. Dabkowski \and Changsong Li \and Jozef H. Przytycki}
\date{}
\maketitle

\begin{abstract}
For a Lattice crossing $L\left( m,n\right) $ we show which Catalan
connection between $2\left( m+n\right) $ points on boundary of $m\times n$
rectangle $P$ can be realized as a Kauffman state and we give an explicit
formula for the number of such Catalan connections. For the case of a
Catalan connection with no arc starting and ending on the same side of the
tangle, we find a closed formula for its coefficient in the Relative
Kauffman Bracket Skein Module of $P\times I$.

\noindent

\textit{{Keywords}: \textup{Knot, Link, Kauffman Bracket, Catalan Tangles}}

\textit{\bigskip }

\textit{\noindent }

\textit{\textup{Mathematics Subject Classification 2000: Primary 57M99;
Secondary 55N, 20D}}
\end{abstract}

\section{Introduction\label{Sec_0}}

We consider Relative Kauffman Bracket Skein Module\footnote{$RKBSM$ was
defined in \cite{JHP} and it was noted there that, $RKBSM$ for $D^{2}\times
I $ with $2k$ points on the boundary is a free module with the basis
consisting of Catalan connections. For an extensive discussions of theory of
Kauffman bracket skein modules see \cite{JP}.} ($RKBSM$) of $P\times I,$
where $P$ is $m\times n$ parallelogram with $2\left( m+n\right) $ points on
the boundary arranged as it is shown in \textrm{Figure} $1.1,$ and $\left(
m+n\right) $-tangle $L\left( m,n\right) $ shown in \textrm{Figure} $1.2$
that we will refer to as $m\times n$\emph{-lattice crossing}. If $\mathfrak{%
Cat}_{m,n}$ denotes the set of all Catalan states for $P$ (crossingless
connections between boundary points) then $L\left( m,n\right) $ in $RKBSM$
of $P\times I$ can be uniquely written in the form $L\left( m,n\right)
=\sum_{C\in \mathfrak{Cat}_{m,n}}r\left( C\right) C,$ where $r\left(
C\right) \in \mathbb{Z}\left[ A^{\pm 1}\right] .$

\begin{figure}[h]
\centering
\begin{minipage}{.35\textwidth}
\centering
\includegraphics[width=\linewidth]{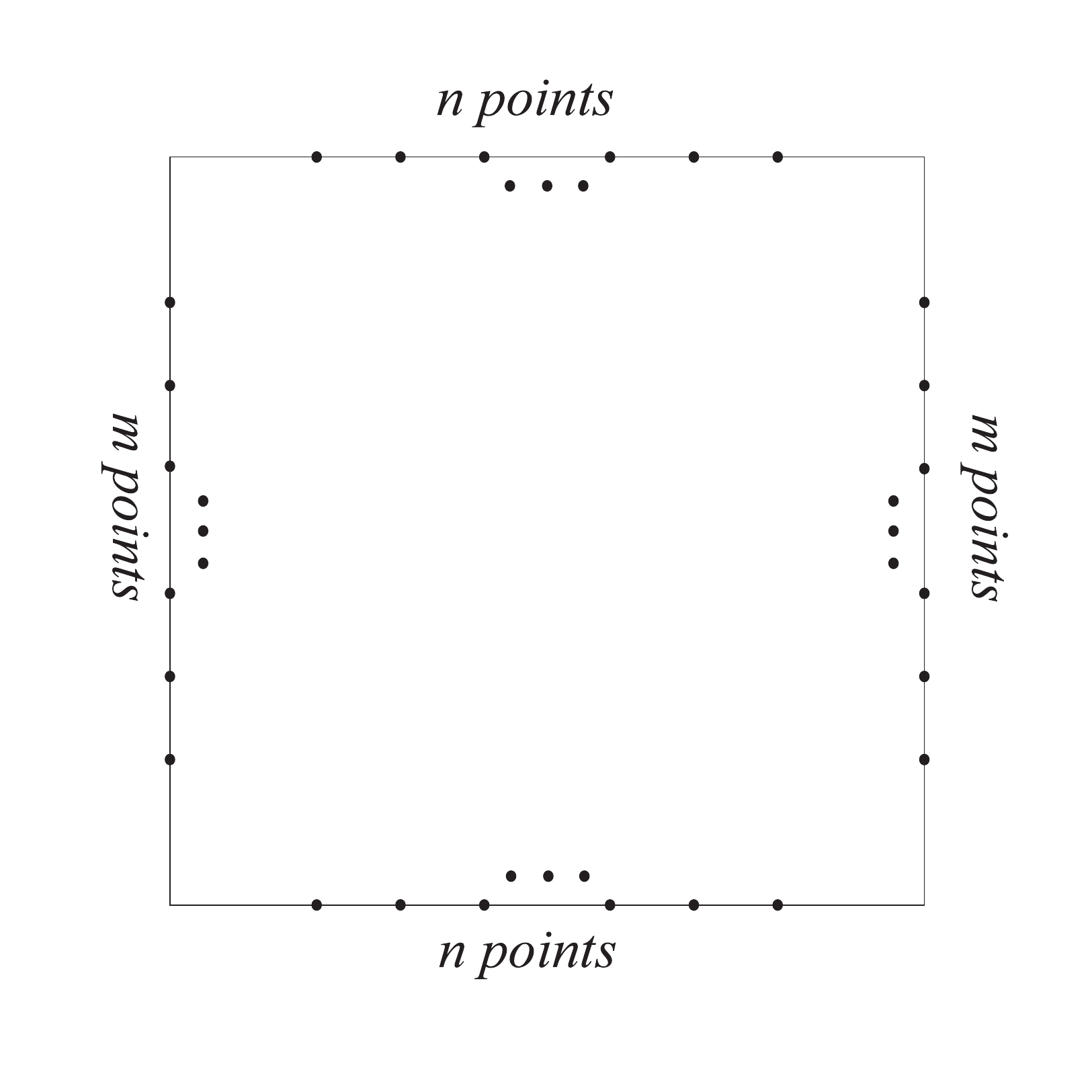}
\caption{Figure 1.1}
\label{fig:test1}
\end{minipage}\hfill 
\begin{minipage}{.35\textwidth}
\centering
\includegraphics[width=\linewidth]{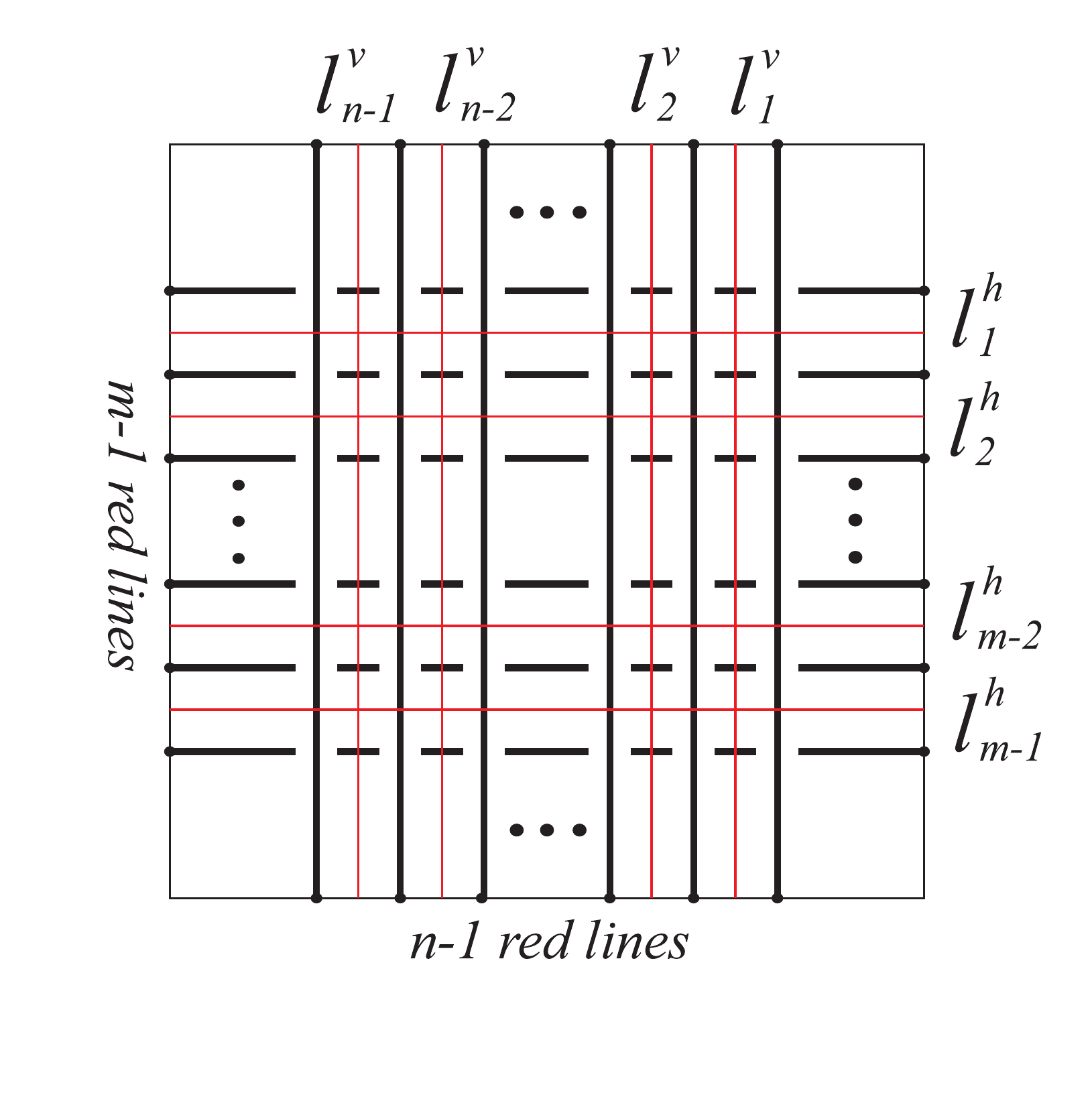}
\caption{Figure 1.2}
\label{fig:test2}
\end{minipage}\hfill 
\begin{minipage}{.2\textwidth}
\centering
\includegraphics[width=\linewidth]{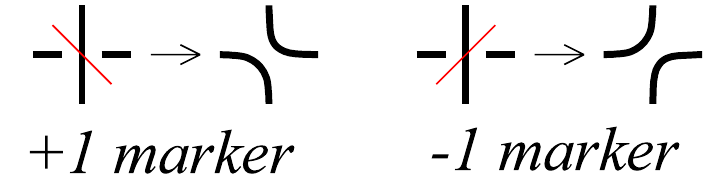}
\caption{Figure 1.3}
\label{fig:test3}
\end{minipage}
\end{figure}

Let $\mathcal{S}_{m,n}$ be the set of all Kauffman states (choices of
positive or negative markers for all crossings as shown in \textrm{Figure} $1
$.$3$), after applying skein relations we have in $RKBSM$ of $P\times I:$%
\begin{equation*}
L\left( m,n\right) =\sum_{s\in \mathcal{S}%
_{m.n}}A^{p(s)-n(s)}(-A^{2}-A^{-2})^{|s|}K\left( s\right) 
\end{equation*}%
where, following standard conventions, for each Kauffman state $s,$ $p(s)$
(resp. $n(s)$) denotes the number of crossings with positive (resp.
negative) marker, $|s|$ the number of trivial components of the diagram
obtained from $L\left( m,n\right) $ by smoothing crossings according to
markers defined by $s,$ and $K\left( s\right) $ denotes the diagram
corresponding to state $s$ after removing trivial components. In $RKBSM$ of $%
P\times I,$ we have%
\begin{equation*}
\sum_{C\in \mathfrak{Cat}_{m,n}}r\left( C\right) C=\sum_{s\in \mathcal{S}%
_{m.n}}A^{p(s)-n(s)}(-A^{2}-A^{-2})^{|s|}K\left( s\right) ,
\end{equation*}%
so, for each Catalan state $C\in \mathfrak{Cat}_{m,n}$ define%
\begin{eqnarray*}
R\left( C\right)  &=&\left\{ s\in \mathcal{S}_{m.n}\text{ }|\text{ }K\left(
s\right) \approx C\right\} ,\text{ where }\approx \text{ denotes isotopy of
diagrams modulo boundary, thus} \\
r\left( C\right)  &=&\left\{ 
\begin{array}{ccc}
0 & if & R\left( C\right) =\emptyset  \\ 
\sum_{s\in R\left( C\right) }A^{p(s)-n(s)}(-A^{2}-A^{-2})^{|s|} & if & 
R\left( C\right) \neq \emptyset 
\end{array}%
\right. .
\end{eqnarray*}%
We call $C$ a \emph{realizable Catalan state} if $R\left( C\right) \neq
\emptyset $ and denote by $\mathfrak{Cat}_{m,n}^{R}$ the set of all
realizable Catalan states. Analogously, we call $C$ a \emph{forbidden} \emph{%
Catalan state }if $R\left( C\right) =\emptyset $ and denote by $\mathfrak{F}%
_{m,n}=\mathfrak{Cat}_{m,n}\backslash \mathfrak{Cat}_{m,n}^{R}$ the set of
all forbidden Catalan states$.$ In this paper we will consider the following
problems:

\begin{description}
\item[1)] Which Catalan states $C$ are realizable as Kauffman states of $%
L\left( m,n\right) ?$

\item[2)] What is an explicit formula for $r\left( C\right) ?$

\item[3)] What is the number of elements of $\mathfrak{Cat}_{m,n}^{R}?$
\end{description}

In the second section we show which Catalan states $C$ can be realized by
Kauffman states of $L\left( m,n\right) $. In the third section, for a
Catalan state $C$ with no arc that starts and ends on the same
\textquotedblleft \emph{side\textquotedblright } of $L\left( m,n\right) $,
we give an explicit formula for the corresponding coefficient $r\left(
C\right) $. Finally, in the forth section, we obtain an explicit formula for
the number of realizable Catalan states.

\section{Realizable and Forbidden Catalan States of $L\left( m,n\right) $%
\label{Sec_1}}

\noindent Enumerating horizontal and vertical lines as in \textrm{Figure} $%
1.2,$ let $l_{1}^{h},$ $l_{2}^{h},$ $...,$ $l_{m-1}^{h}$ be \emph{horizontal
lines} and $l_{1}^{v},$ $l_{2}^{v},$ $...,$ $l_{n-1}^{v}$ be \emph{vertical
lines}. For a Catalan state $C\in \mathfrak{Cat}_{m,n},$ we let%
\begin{equation*}
d^{h}\left( C\right) =\max \left\{ \#\left( l_{i}^{h}\cap C\right) \text{ }|%
\text{ }i=1,2,\text{ }...,\text{ }m-1\right\} \text{ and }d^{v}\left(
C\right) =\max \left\{ \#\left( l_{i}^{v}\cap C\right) \text{ }|\text{ }%
i=1,2,\text{ }...,\text{ }n-1\right\} 
\end{equation*}%
where $\#\left( l_{i}^{\alpha }\cap C\right) $ \ is the minimal number of
intersections of $l_{i}^{\alpha }$ and $C$ for $\alpha =h$ or $\alpha =v.$

\begin{lemma}
\label{Lemma 2.1}Let $C\in \mathfrak{Cat}_{m,n}$ and assume that $%
d^{h}\left( C\right) >n$ or $d^{v}\left( C\right) >m$ then $C\in \mathfrak{F}%
_{m,n}.$
\end{lemma}

\begin{proof}
Observe that for any Kauffman state $s$ of $L\left( m,n\right) ,$
corresponding diagram $L^{s}\left( m,n\right) $ obtained by smoothing
crossings according to choice of markers defined by $s$ cuts horizontal and
vertical separating lines precisely $n$ times and $m$ times. Hence for the
Catalan state $C\left( s\right) $ obtained from $L^{s}\left( m,n\right) $ we
must have $d^{h}\left( C\right) \leq n$ and $d^{v}\left( C\right) \leq m$.
\end{proof}

\smallskip 

We say that a Catalan state $C$ satisfies \emph{horizontal-line condition,} $%
H\left( m,n\right) $ if $d^{h}\left( C\right) \leq n$ and, correspondingly, $%
C$ satisfies \emph{vertical-line condition,} $V\left( m,n\right) $ if $%
d^{v}\left( C\right) \leq m.$ We show that if $C\in \mathfrak{Cat}_{m,n}$
satisfies both conditions $H\left( m,n\right) $ and $V\left( m,n\right) $
then $C$ is a realizable Catalan state. The following lemma gives sufficient
conditions for a Catalan state $C\in \mathfrak{Cat}_{m,n}$ to meet condition 
$V\left( m,n\right) $ or/and $H\left( m,n\right) $.

\begin{lemma}
\label{Lemma 2.2}For Catalan state $C,$

\begin{description}
\item[$\left( 1\right) $] if $C$ has no returns on the top horizontal
boundary then $C$ satisfies condition $V\left( m,n\right) ;$

\item[$\left( 2\right) $] if $C,$ in addition, has no returns on the right
side, then $C$ satisfies both $H\left( m,n\right) $ and $V\left( m,n\right) $%
.
\end{description}
\end{lemma}

\begin{proof}
For $\left( 1\right) ,$ we argue by case distinction. Let $x_{1},...,x_{n}$
be boundary points on top of $C$, and $\ell $ be a separating line between $%
x_{k}$ and $x_{k+1}$ (see \textrm{Figure} $2.1$).

If $x_{k+i}$th connection cuts $\ell $ for some $1\leq i\leq n-k,$ then all $%
x_{1},...,x_{k}$ connections do not cut $\ell $. Since there are $m+2k$
points on the right of $\ell $ and among them at least $2k$ points yield
connections in $C$ without cutting $\ell $, it follows that there are at
most $m$ points on the right of $\ell $ which can form connections with
points on the left of $\ell $. Using similar reasoning, if $1\leq i\leq k$
then at most $m$ points on the left of $\ell $ can form connections with
points on the right of $\ell $. Therefore, statement $\left( 1\right) $
follows from basic counting.

If each $x_{i}$ connection does not cut $\ell $ for all $i=1,$ $2,$ $...,$ $%
n,$ we observe that, in particular, $x_{1},...,x_{k}$ connections do not cut 
$\ell .$ Thus, as in the previous case, we argue that at most $m$ points on
the left of $\ell $ can be connected with points on the right of $\ell $.

Statement $\left( 2\right) ,$ follows directly from $(1),$ since $(1)$ can
be applied to any side of the rectangle.
\end{proof}

\smallskip 
\begin{figure}[h]
\centering
\begin{minipage}{.35\textwidth}
\centering
\includegraphics[width=\linewidth]{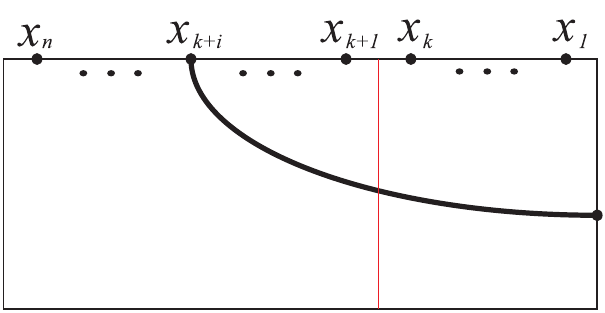}
\caption{Figure 2.1}
\label{fig:test4}
\end{minipage}\qquad 
\begin{minipage}{.27\textwidth}
\centering
\includegraphics[width=\linewidth]{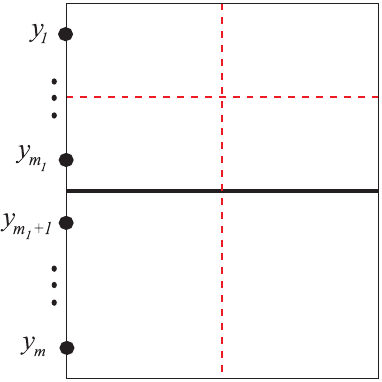}
\caption{Figure 2.2}
\label{fig:test5}
\end{minipage}
\end{figure}

For two Catalan states $C_{1}\in \mathfrak{Cat}_{m_{1},n}$ and $C_{2}\in 
\mathfrak{Cat}_{m_{2},n}$ with no returns on the bottom of the first and the
top for the second one defines an operation $\ast _{v}$ of vertical
composition of Catalan states (see \textrm{Figure} $2.2$). The operation $%
\ast _{v}$ allows us to define Catalan state $C=C_{1}\ast _{v}C_{2},$ for
which the number of intersections with its vertical lines can easily be
counted provided that the number of intersections with vertical lines of the
factor states is known. Analogously, one defines horizontal composition $%
\ast _{h}$ of Catalan states.

\begin{lemma}
\label{Lemma 2.3}Let $C_{1}\in \mathfrak{Cat}_{m_{1},n}$ and $C_{2}\in 
\mathfrak{Cat}_{m_{2},n}$ and assume that bottom of $C_{1}$ and top of $%
C_{2} $ are no return sides respectively $($see \textrm{Figure} $2.2)$. Then
the vertical composition $C=C_{1}\ast _{v}C_{2}$ of states $C_{1}$ and $%
C_{2} $ is a Catalan state with the following properties$:$

\begin{description}
\item[$\left( 1\right) $] Every vertical line $\ell $ intersects $C $ at
most $m=m_{1}+m_{2}$ times and every horizontal line $\ell ^{\prime }$
either in $C_{1}$ or $C_{2}$ cuts $C$ the same number of times as in $C_{1}$
or $C_{2}$ respectively$.$

\item[$\left( 2\right) $] Catalan states $C_{1}$ and $C_{2}$ satisfy
conditions $H\left( m_{i},n\right) $ and $V\left( m_{i},n\right) $, for $%
i=1,2,$ respectively if and only if $C$ satisfies conditions $H\left(
m,n\right) $ and $V\left( m,n\right) $.
\end{description}
\end{lemma}

\begin{proof}
Since $C_{1}$ and $C_{2}$ have no return sides on the bottom and the top
respectively, vertical composition $C=C_{1}\ast _{v}C_{2}$ of states $C_{1}$
and $C_{2}$ (as it is shown in \textrm{Figure} $2.2$) is a Catalan state.
Furthermore, if $\ell $ is a vertical line in $C,$ then from Lemma \ref%
{Lemma 2.2}$(1)$ it follows that the number of times that $\ell $ cuts $C$
is bounded above by the number $m_{1}+m_{2}=m$. Finally, if $\ell ^{\prime }$
is a horizontal line either in $C_{1}$ or $C_{2}$ then $\ell ^{\prime }$ is
also a separating line in the composed state $C$ and, since $C_{1}$ and $%
C_{2}$ are composed along sides with no returns, the number of times $\ell
^{\prime }$ cuts $C$ is identical as the number of times $\ell ^{\prime }$
cuts either $C_{1}$ or $C_{2}$. This finishes our proof of $\left( 1\right) .
$

Statement $\left( 2\right) $ follows from $(1)$.
\end{proof}

\smallskip 

For Catalan states $C$ satisfying conditions $H\left( m,n\right) $ and $%
V\left( m,n\right) ,$ we apply induction on $m+n$ to prove that such states
could be obtained as Kauffman states of $L\left( m,n\right) .$ An important
step in our inductive proof is an observation that Catalan states with $%
d^{h}\left( C\right) =n$ (or $d^{v}\left( C\right) =m$) could be obtained as
a vertical (horizontal) composition of Catalan states.

\begin{lemma}[Split Lemma]
\label{Lemma 2.4}Let $C\in \mathfrak{Cat}_{m,n}$ and assume that $C$
satisfies conditions $H\left( m,n\right) $ and $V\left( m,n\right) ,$ and $%
d^{h}\left( C\right) =n$. Then

\begin{description}
\item[$\left( 1\right) $] $C=C_{1}\ast _{v}C_{2},$ where $C_{1}\in \mathfrak{%
Cat}_{m_{1},n}$, $C_{2}\in \mathfrak{Cat}_{m_{2},n}$, $m_{1}+m_{2}=m,$ and

\item[$\left( 2\right) $] Catalan states $C_{1}$ and $C_{2}$ satisfies
conditions $H\left( m_{i},n\right) $ and $V\left( m_{i},n\right) $, for $%
i=1,2,$ respectively.
\end{description}
\end{lemma}

\begin{proof}
Since $d^{h}\left( C\right) =n$, there is a horizontal separating line $\ell 
$ that cuts $C$ precisely $n$ times. Splitting $C$ along $\ell $ yields a
pair of Catalan states $C_{1}\in \mathfrak{Cat}_{m_{1},n}$ and $C_{2}\in 
\mathfrak{Cat}_{m_{2},n}$, where $m_{1}+m_{2}=m.$ Since $n$ is the number of
intersections of $\ell $ with Catalan state $C,$ it follows that $C_{1}$ has
the bottom side with no returns and $C_{2}$ has the top side with no
returns. Therefore, state $C$ can be expressed as $C=C_{1}\ast _{v}C_{2}$.
Since $C$ satisfies conditions $H\left( m,n\right) $ and $V\left( m,n\right)
,$ it follows from Lemma \ref{Lemma 2.3}$\left( 2\right) $ that each $C_{i}$
satisfies conditions $H\left( m_{i},n\right) $ and $V\left( m_{i},n\right) $%
, for $i=1,2,$ respectively.
\end{proof}

\smallskip

Obviously, the statement of \emph{Split Lemma} holds when $d^{v}\left(
C\right) =m$. In such a case, $C=C_{1}\ast _{h}C_{2},$ where $C_{1}\in 
\mathfrak{Cat}_{m,n_{1}}$, $C_{2}\in \mathfrak{Cat}_{m,n_{2}}$, $%
n_{1}+n_{2}=n,$ and each $C_{i}$ satisfies conditions $H\left(
m,n_{i}\right) $ and $V\left( m,n_{i}\right) $, for $i=1,2,$ respectively.
Now, we state and prove the main result for this section.

\begin{theorem}
\label{Theorem 2.5}Assume that $C$ satisfies conditions $H\left( m,n\right) $
and $V\left( m,n\right) $, then $C$ can be obtained as a Kauffman state of $%
L\left( m,n\right).$
\end{theorem}

\begin{proof}
We proceed by induction on $m+n$. Clearly, statement is vacuously true for $%
m=n=0$, and it is obvious when $m=n=1$. We assume that statement is true for
all pairs $\left( m,n\right) $ such that $m,$ $n>0$ and $m+n\leq k,$ where $%
k\geq 2.$ We show that statement is true for $m,$ $n>0$ and $m+n=k+1$ and
WLOG we can assume $m\geq 2$.

\noindent Let $C\in \mathfrak{Cat}_{m,n}$ and $C$ satisfies conditions $%
H\left( m,n\right) $ and $V\left( m,n\right) $. We have the following cases

\begin{description}
\item[1)] $d^{h}\left( C\right) =n$ or $d^{v}\left( C\right) =m$
\end{description}

If $d^{h}\left( C\right) =n,$ then applying \emph{Split Lemma,} we have $%
C=C_{1}\ast _{v}C_{2}$, where $C_{1}\in \mathfrak{Cat}_{m_{1},n}$, $C_{2}\in 
\mathfrak{Cat}_{m_{2},n}$ $\left( m_{1}+m_{2}=m\right) ,$ and each $C_{i}$
satisfies conditions $H\left( m_{i},n\right) $ and $V\left( m_{i},n\right) $%
, for $i=1,$ $2$. Since $m_{i}+n\leq k$, applying induction hypothesis, we
obtain corresponding Kauffman states $s_{1}$ and $s_{2}$ for $L\left(
m_{1},n\right) $ and $L\left( m_{2},n\right) $ respectively such that $%
C_{1}\approx K\left( s_{1}\right) $ and $C_{2}\approx K\left( s_{2}\right) $%
. Let $s$ be Kauffman state obtained by putting markers according to $s_{1}$
in the first $m_{1}$ rows of $L\left( m,n\right) $ and according to $s_{2}$
in the remaining $m-m_{1}=m_{2}$ rows of $L\left( m,n\right) $. Clearly, for
Kauffman state $s,$ we have $C\approx K\left( s\right) $.

Proof in the case when $d^{v}\left( C\right) =m$ is analogous.

\begin{description}
\item[2)] $d^{h}\left( C\right) <n$ and $d^{v}\left( C\right) <m,$ that is,
each horizontal line cuts $C$ less than $n$ times and each vertical line
cuts $C$ less than $m$ times.
\end{description}

In this case we will construct a Kauffman state $s$ such that $C\approx
K\left( s\right) $ using an algorithm given in the following lemma.

\begin{lemma}
\label{Lemma 2.6}Assume that $d^{h}\left( C\right) <n$ and $d^{v}\left(
C\right) <m$, then $C$ can be obtained as a Kauffman state of $L\left(
m,n\right) .$
\end{lemma}

\begin{proof}
By assumption, each horizontal line cuts $C$ less than $n$ times and each
vertical line cuts $C$ less than $m$ times, so the top boundary side of $C$
must have at least one return (see \textrm{Figure} $2.3$ (left)), or at
least one corner connection (see \textrm{Figure} $2.3$ (right)). If $C$ has
no corner connection, we start construction by defining Kauffman state $s_{1}
$ (first row of desired Kauffman state $s$) for the returning connection of $%
C$ between pair of points $x_{i}$ and $x_{i+1}$ on the top of $C$ with
minimal index $i$, $i=1,2,...,n-1$ as shown in \textrm{Figure} $2.4$ (left).
If there is a corner connection, we take as Kauffman state $s_{1}$ (first
row of Kauffman state $s$) the state shown in \textrm{Figure} $2.4$ (right).

\begin{figure}[h]
\centering
\begin{minipage}{.45\textwidth}
\centering
\includegraphics[width=\linewidth]{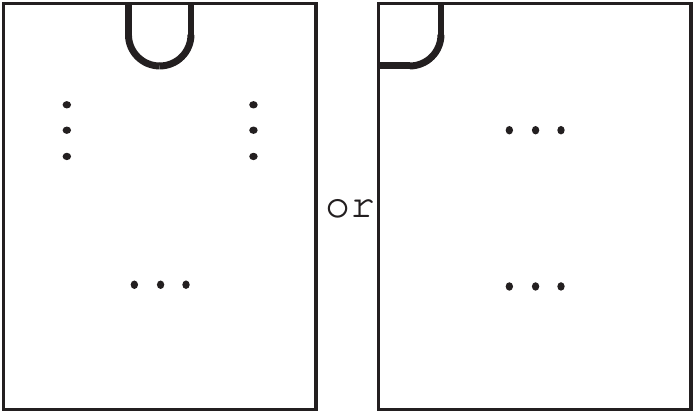}
\caption{Figure 2.3}
\label{fig:test6}
\end{minipage}\qquad 
\begin{minipage}{.45\textwidth}
\centering
\includegraphics[width=\linewidth]{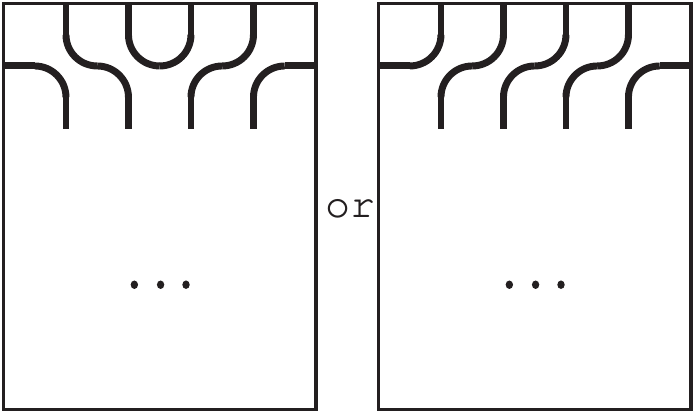}
\caption{Figure 2.4}
\label{fig:test7}
\end{minipage}
\end{figure}

Using isotopy, we deform state $C$ into an equivalent state $\widetilde{C}$
with top as shown in \textrm{Figure} $2.4$ and, for which all remaining
connections form $\left( \left( m-1\right) +n\right) $-Catalan state$.$
Since every vertical line cuts $C$ less than $m$ times, we observe that it
must cut each vertical line no more then $m-2$ times. It follows that $%
\left( \left( m-1\right) +n\right) $-Catalan state $C^{\prime }\in \mathfrak{%
Cat}_{m-1,n}$ obtained from $\widetilde{C}$ by removing its top (see \textrm{%
Figure} $2.4$) cuts all vertical separating lines no more than $m-1$ times.
In particular, Catalan state $C^{\prime }$ satisfies conditions $H\left(
m-1,n\right) $ and $V\left( m-1,n\right) $ and one may apply induction
hypothesis for the state $C^{\prime }$. Therefore, there is a Kauffman state 
$s^{\prime }$ of $L\left( m-1,n\right) $ such that $K\left( s^{\prime
}\right) =C^{\prime }$. In order to complete our construction, we take for $s
$ the state obtained by putting markers according to state $s_{1}$ in the
first row of $L\left( m,n\right) $ and putting markers according to $%
s^{\prime }$ in its remaining $\left( m-1\right) $ rows. For the choice of
markers determined by $s$, we clearly have $K\left( s\right) \approx C.$
\end{proof}

\smallskip

Proof of Lemma \ref{Lemma 2.6} also completes our proof of Theorem \ref%
{Theorem 2.5}.
\end{proof}

\smallskip

\section{Coefficients of Realizable Catalan States with No Returns in $RKBSM$%
\label{Sec_3}}

In this section we prove result concerning coefficients of realizable
Catalan states for $L\left( m,n\right) $ for the case of Catalan states with
no returns. This result was proved in $1992$ by J. Hoste and J. H. Przytycki%
\footnote{%
Authors of \cite{H-P-2} didn't publish this result. We were informed by
Mustafa Hajij that a related formula was noted by S. Yamada \cite{SY},\cite%
{Haj}.} while they were working on the Kauffman bracket skein module of the
Whitehead type manifolds.

\begin{theorem}
\label{Theorem 3.1}In $RKBSM$ of $D^{2}\times I$ with $2\left( m+n\right) $
points we have%
\begin{equation*}
L\left( m,n\right) =\sum_{k=0}^{m}P_{k,m-k,h}T_{k,\,m-k,\text{ }h}+Remainder,
\end{equation*}%
where $P_{k,m-k,h}=A^{nm-2\left( k+h\right) k}{m\brack k}_{A^{-4}};$ $%
T_{k,m-k,h}$ is the Catalan state with $2k$ positive diagonal arcs, $2\left(
m-k\right) $ negative diagonal arcs and $h=n-m\geq 0$ vertical arcs $($see 
\textrm{Figure} $3.1);$ and the \emph{remainder} is a polynomial $($over $%
\mathbb{Z}\left[ A^{\pm 1}\right] $ $)$ in variables consisting of
realizable Catalan states that have at least one return.\footnote{%
It is worth to observe that if we decorate lattice crossing $L\left(
m,n\right) $ by Jones-Wenzl idempotents $f_{m}$ and $f_{n}$ (respectively)
then reminder is $0$ since $f_{k}e_{i}=0=e_{i}f_{k}$ for any $i$ and $k$.}
\end{theorem}

\smallskip

\begin{figure}[h]
%  figure placement: here, top, bottom, or page
\centering
\includegraphics[scale=0.75]{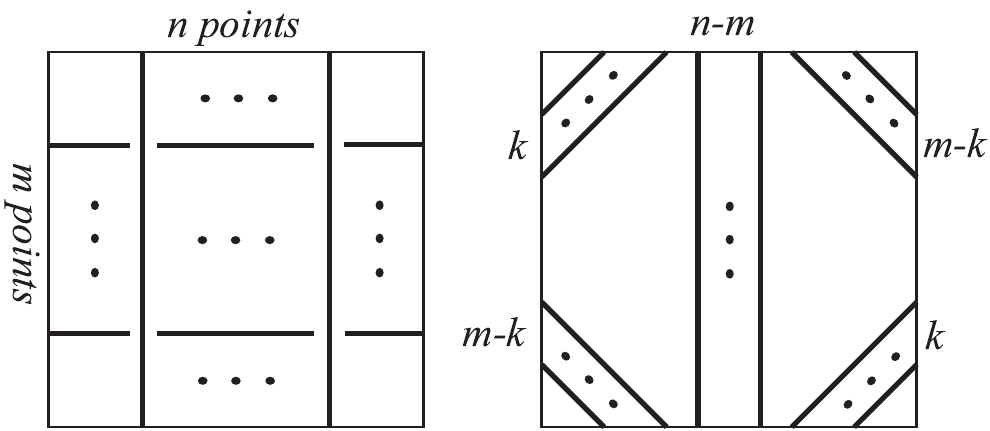}  
\caption{Figure 3.1}
\label{Figure 3.1}
\end{figure}

\noindent Before we give a proof of Theorem \ref{Theorem 3.1}, we need to
recall the notation $q$-analogue of an integer $m$ and $q$-analogue of
binomial coefficient$.$ A $q$-analogue of an integer $m$ is defined as
follows $[m]_{q}=1+q+q^{2}+...+q^{m-1},$ and if $[m]_{q}!=\left[ 1\right]
_{q}\cdot \left[ 2\right] _{q}\cdot ...\cdot \left[ m\right] _{q}$ then $q$%
-analogue of binomial coefficient is defined by%
\begin{equation*}
{m\brack k}_{q}=\frac{[m]_{q}!}{[k]_{q}!\cdot \lbrack m-k]_{q}}
\end{equation*}%
One shows using the following identity for $q$-analogue of binomial
coefficient%
\begin{equation*}
{m\brack k}_{q}={m-1\brack k}_{q}+q^{m-k}{m-1\brack k-1}_{q}
\end{equation*}%
and induction that ${m\brack k}_{q}\in \mathbb{Z}\left[ q\right] $.

\begin{proof}
Applying Kauffman bracket skein relation (see \textrm{Figure} $1.3$) for the
crossing in top left corner of $L\left( m,n\right) $ (see \textrm{Figure} $%
3.2$) we obtain the following recursive formula $P_{k,m-k,h}=A^{\left(
m+n\right) -1}P_{k,m-k-1,h}+A^{1-\left( m+n\right) }P_{k-1,m-k,h}:$

\begin{figure}[h]
%  figure placement: here, top, bottom, or page
\centering
\includegraphics[scale=0.65]{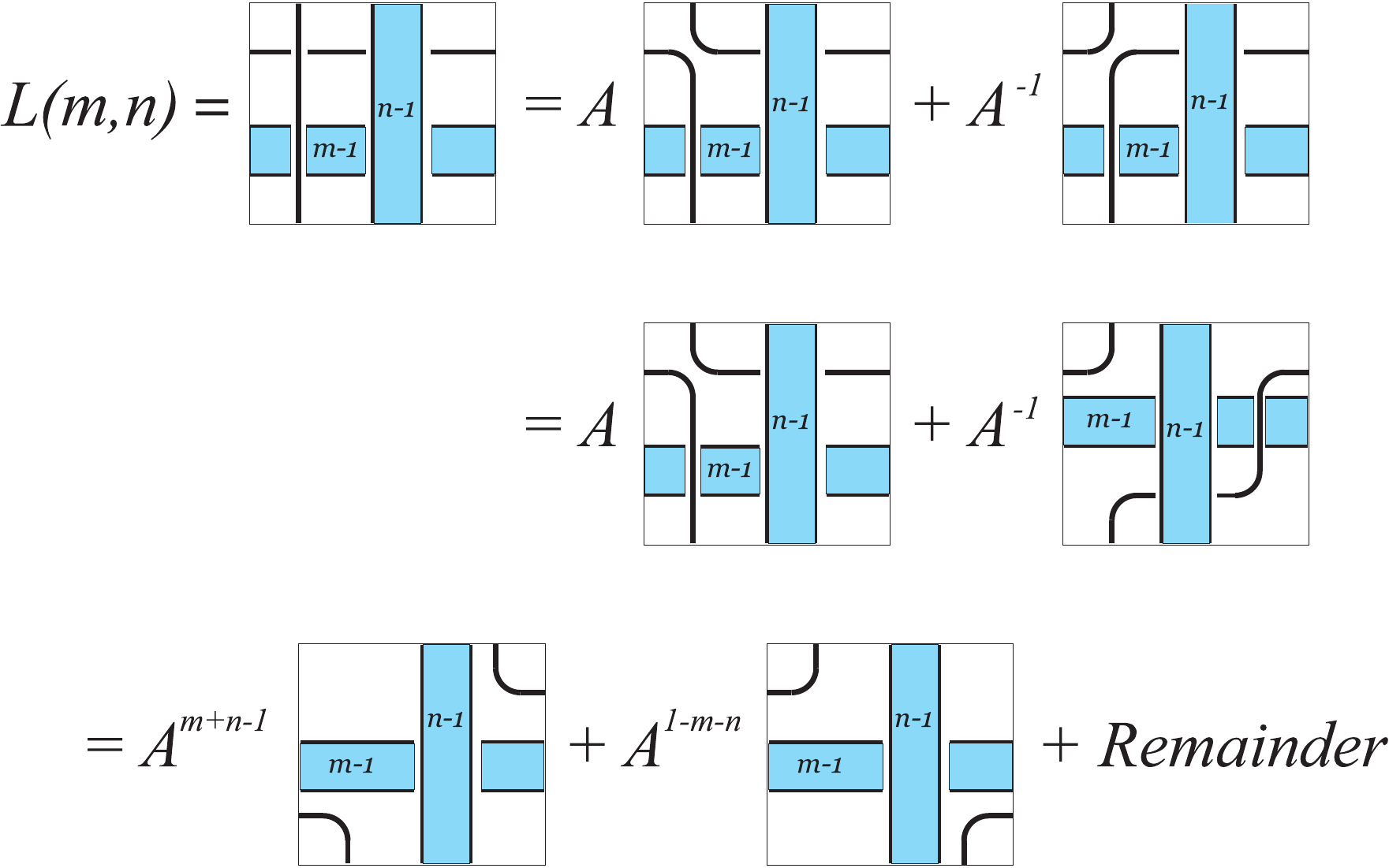}  
\caption{Figure 3.2}
\label{Figure 3.2}
\end{figure}

Now, we let $P_{m,0,h}=A^{-mn}$ and $P_{0,m,h}=A^{mn}.$ One verifies that $%
P_{k,m-k,h}=A^{mn-2\left( k+h\right) k} {m \brack k}_{A^{-4}}$ satisfies the
above recurrence\footnote{%
For the initial conditions we check:
\par
$P_{m,0,h}=A^{mn-2\left( m+h\right) m}=A^{mn-2\left( m+n-m\right)
m}=A^{-mn}, $ and $P_{0,m,h}=A^{mn-0}=A^{mn}$.}. Namely, we have by the
property of $q$-analogue of binomial coefficient with $q=A^{-4}:$%
\begin{eqnarray*}
P_{k,m-k,h} &=&A^{mn-2\left( k+h\right) k} {m \brack k}_{A^{-4}}=A^{mn-2%
\left( k+h\right) k} {m-1 \brack k}_{A^{-4}}+A^{mn-2\left( k+h\right)
k}A^{-4\left( m-k\right) } {m-1 \brack k-1}_{A^{-4}} \\
&=&A^{\left( m+n\right) -1}A^{1-m-n+mn-2(k+h)k} {m-1 \brack k}%
_{A^{-4}}+A^{1-\left( m+n\right) }A^{m+n-1+mn-2(k+h)k-4(m-k)} {m-1 \brack k-1%
}_{A^{-4}} \\
&=&A^{\left( m+n\right) -1}A^{\left( m-1\right) \left( n-1\right) -2(k+h)k}{%
m-1 \brack k}_{A^{-4}}+A^{1-\left( m+n\right)
}A^{(m-1)(n-1)+2((k-1)+h)(k-1)} {m-1 \brack k-1}_{A^{-4}} \\
&=&A^{\left( m+n\right) -1}P_{k,m-k-1,h}+A^{1-\left( m+n\right)
}P_{k-1,m-k,h}
\end{eqnarray*}%
as needed.
\end{proof}

\smallskip

\begin{remark}
One of the goals for future research is to find a "\emph{relatively good}"
formula for coefficients of all realizable Catalan states of $L\left(
m,n\right) $ in $RKBSM$ of $D^{2}\times I$ with $2\left( m+n\right) $
points. In particular, Theorem \ref{Theorem 3.1} gives an explicit formula
for all coefficients of Catalan states with no returns. A natural
generalization of this family Catalan states would be families of Catalan
states with three, two, or just one side that has no return. As we proved in
Lemma \ref{Lemma 2.2} even if only one side $($left or right$)$ of the
Catalan state has no return then any horizontal separating line can cut such
a Catalan state in at most $n$ points.
\end{remark}

\section{Number of Realizable Catalan States by $L\left( m,n\right) $\label%
{Sec_4}}

In this section, we find a closed formula for $\left\vert \mathfrak{Cat}%
_{m,n}^{R}\right\vert $ -- the number of Catalan states that could be
obtained as Kauffman states of $L\left( m,n\right) ,$ or alternatively, we
find number of elements of the set $\mathfrak{F}_{m,n}=\mathfrak{Cat}%
_{m,n}\backslash \mathfrak{Cat}_{m,n}^{R}$ of all \emph{Forbidden Catalan
states.} As we showed in Theorem \ref{Theorem 2.5}, a Catalan state $C\in 
\mathfrak{Cat}_{m,n}$ is realizable if and only if $C$ satisfies conditions $%
H\left( m,n\right) $ and $V\left( m,n\right) .$ Therefore, every forbidden
state $C$ must intersects at least one of the horizontal or vertical lines
more than $n$ or $m$ times respectively. Let $\mathbf{T}^{h}\left(
m,n\right) $ and $\mathbf{T}^{v}\left( m,n\right) $ denote sets of Catalan
states that are excluded by horizontal and vertical condition respectively.
More precisely, we have for the lines in \textrm{Figure} $1.2:$ 
\begin{equation*}
\mathbf{T}^{h}\left( m,n\right) =\left\{ C\in \mathfrak{Cat}_{m,n}\text{ }|%
\text{ }d^{h}\left( C\right) >n\right\} \text{ and }\mathbf{T}^{v}\left(
m,n\right) =\left\{ C\in \mathfrak{Cat}_{m,n}\text{ }|\text{ }d^{v}\left(
C\right) >m\right\} 
\end{equation*}%
We call them respectively \emph{horizontally excluded} and \emph{vertically
excluded Catalan states}. Clearly, we see that%
\begin{equation*}
\mathbf{T}^{h}\left( m,n\right) \cup \mathbf{T}^{v}\left( m,n\right) =%
\mathfrak{F}_{m,n}.
\end{equation*}%
However, more is true, i.e. $\mathbf{T}^{h}\left( m,n\right) \cap \mathbf{T}%
^{v}\left( m,n\right) =\emptyset $, which is a consequence of the following
proposition.

\begin{proposition}
\label{Proposition 4.1}Let $S$ be a circle with an even number of points $%
a_{1}+a_{2}+a_{3}+a_{4}$ and consider a pair of lines intersecting at one
point and splitting $S$ into sectors $A_{i}$ with $a_{i}$ points each
sectors, $i=1,$ $2,$ $3,$ $4$ $($see \textrm{Figure} $4.1)$. Let $C$ be a
Catalan state and $m\left( C\right) $ be the number of intersections of $C$
with the pair of lines.

\begin{description}
\item[(i)] If points from sectors $A_{2}$ and $A_{4}$ are not connected in $%
C,$ then $m\left( C\right) \leq a_{1}+a_{3}.$

\item[(ii)] If points from sectors $A_{1}$ and $A_{3}$ are not connected in $%
C,$ then $m\left( C\right) \leq a_{2}+a_{4}.$

\item[(iii)] For any $C,$ $m\left( C\right) \leq \min \left\{
a_{1}+a_{3},\,a_{2}+a_{4}\right\} $.
\end{description}
\end{proposition}

\begin{proof}
We observe that in a Catalan state $C$ we cannot have points from $A_{1}$
connected with points from $A_{3},$ and points from $A_{2}$ connected with
points from $A_{4}$ at the same time. For $\left( \mathbf{i}\right) ,$ we
notice that since points from sectors $A_{2}$ and $A_{4}$ are not connected
in $C,$ if $u$ is the number of connections in $C$ between points from
sectors $A_{1}$ and $A_{3}$ then these connections intersect pair of lines
at $2u$ points and involve $2u$ points from sectors $A_{1}$ and $A_{3}$ in
total. Consequently, there is a total of $a_{1}+a_{3}-2u$ points left in
sectors $A_{1}$ and $A_{3}$ that may form connections intersecting the pair
of lines at most once each. It follows that total number of intersections of
Catalan state $C$ with the pair of lines is bounded above by $a_{1}+a_{3}$.

Proof of $\left( \mathbf{ii}\right) $ follows by changing the roles of $%
A_{1},$ $A_{3}$ and $A_{2},$ $A_{4},$ and statement $\left( \mathbf{iii}%
\right) $ is a direct consequence of $\left( \mathbf{i}\right) $ and $\left( 
\mathbf{ii}\right) $.
\end{proof}

\begin{figure}[h]
%  figure placement: here, top, bottom, or page
\centering
\includegraphics[scale=0.22]{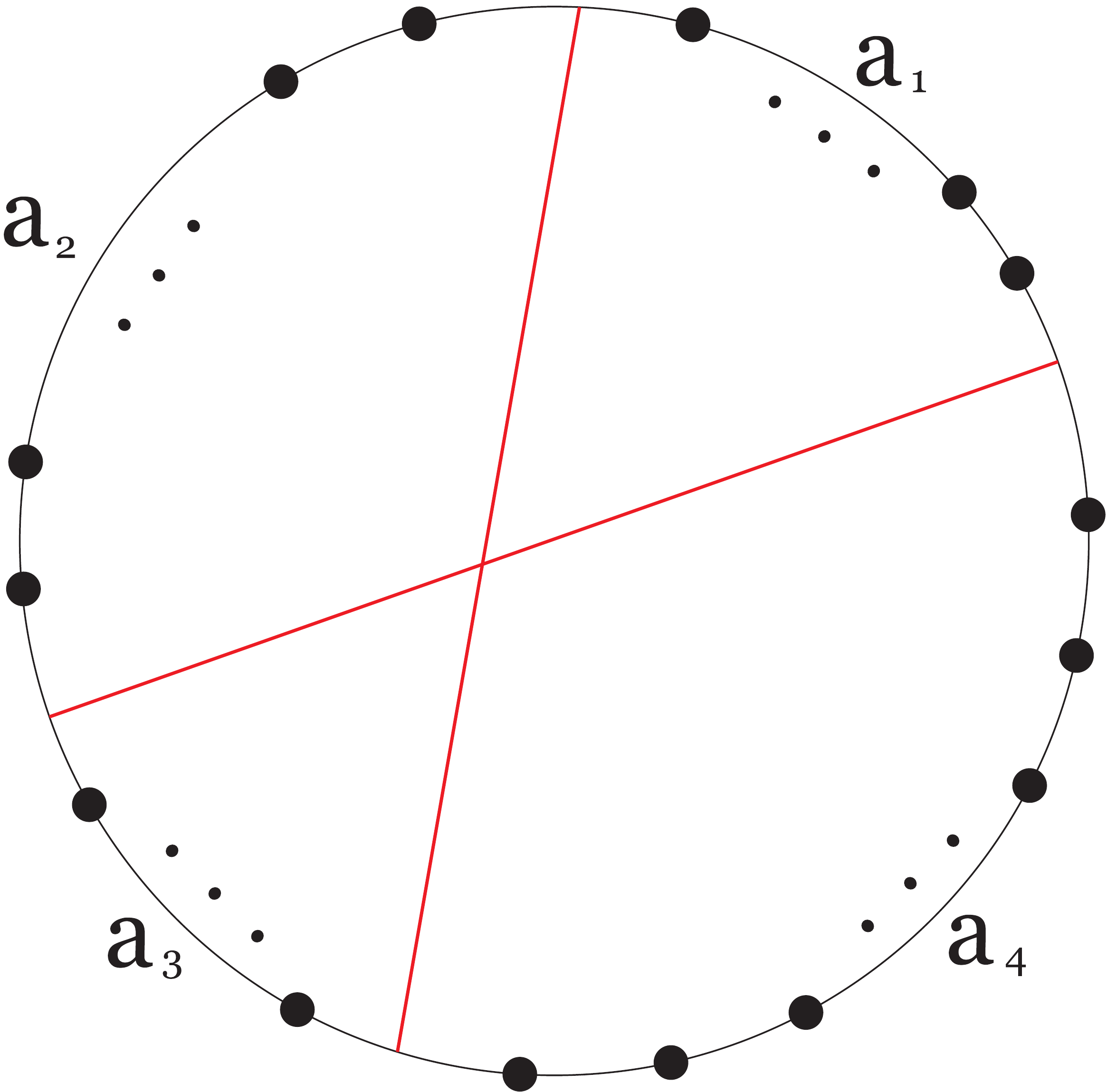}  
\caption{Figure 4.1}
\label{Figure 4.1}
\end{figure}

Proposition \ref{Proposition 4.1} has the following important consequence.

\begin{corollary}
\label{Corollary 4.2}Horizontally and vertically excluded Catalan states are
disjoint.
\end{corollary}

\begin{proof}
For a pair of lines consisting of a vertical and a horizontal line as in 
\textrm{Figure} $1.2,$ we see that $a_{1}+a_{3}=a_{2}+a_{4}=m+n.$ Applying
result of Proposition \ref{Proposition 4.1} for such a pair of lines, for
any Catalan state $C,$ we have $m\left( C\right) \leq m+n.$ Therefore, if
the horizontal line intersects $C$ more then $n$ times then the vertical
line must intersect $C$ less than $m$ times. Analogously, if the vertical
line intersects $C$ more then $m$ times then the horizontal line must
intersect $C$ less than $n$ times.
\end{proof}

\medskip

As a consequence of Corollary \ref{Corollary 4.2} and Theorem \ref{Theorem
2.5}, we obtain the following result.

\begin{corollary}
\label{Corollary 4.3}Let $T^{h}\left( m,n\right) $ and $T^{v}\left(
m,n\right) $ be the number of horizontally and vertically excluded Catalan
states. The number of Forbidden Catalan states equals $T^{h}\left(
m,n\right) +T^{v}\left( m,n\right).$
\end{corollary}

It is worth to note that following relation between numbers $T^{h}\left(
m,n\right) $ and $T^{v}\left( m,n\right) $ holds%
\begin{equation*}
T^{h}\left( m,n\right) =T^{v}\left( n,m\right) .
\end{equation*}%
Consequently, from now on, we will consider only $T^{h}\left( m,n\right) $.

In order to find $T^{h}\left( m,n\right) ,$ we start by constructing
bijection between the set $\mathfrak{Cat}_{m,n}$ and the set of Dyck paths $%
\mathfrak{D}_{m,n}$ of semilength $m+n$ from $\left( 0,0\right) $ to $\left(
m+n,\text{ }m+n\right) $ that never pass below the line $y=x$ and consists
only the vertical steps $V=\left( 0,1\right) $ (\emph{ascents}) and
horizontal steps $H=\left( 1,0\right) $ (\emph{descents}) as it is shown in 
\textrm{Figure} $4.2$.

Consider $2\left( m+n\right) +1$ lines $l_{0},$ $l_{1},$ $...,$ $l_{2\left(
m+n\right) }$ as in \textrm{Figure} $4.3$. Any Catalan state $C\in \mathfrak{%
Cat}_{m,n}$ uniquely determines the minimal number of intersections $%
n_{j}\left( C\right) =\#l_{j}\cap C$ of $C$ with each line $l_{j},$ $0\leq
j\leq 2\left( m+n\right) $ and let 
\begin{equation*}
r_{j}\left( C\right) =n_{j}\left( C\right) -n_{j-1}\left( C\right) ,\text{ }%
j=1,2,...,2\left( m+n\right) .
\end{equation*}%
\begin{figure}[h]
\centering
\begin{minipage}{.45\textwidth}
\centering
\includegraphics[width=\linewidth]{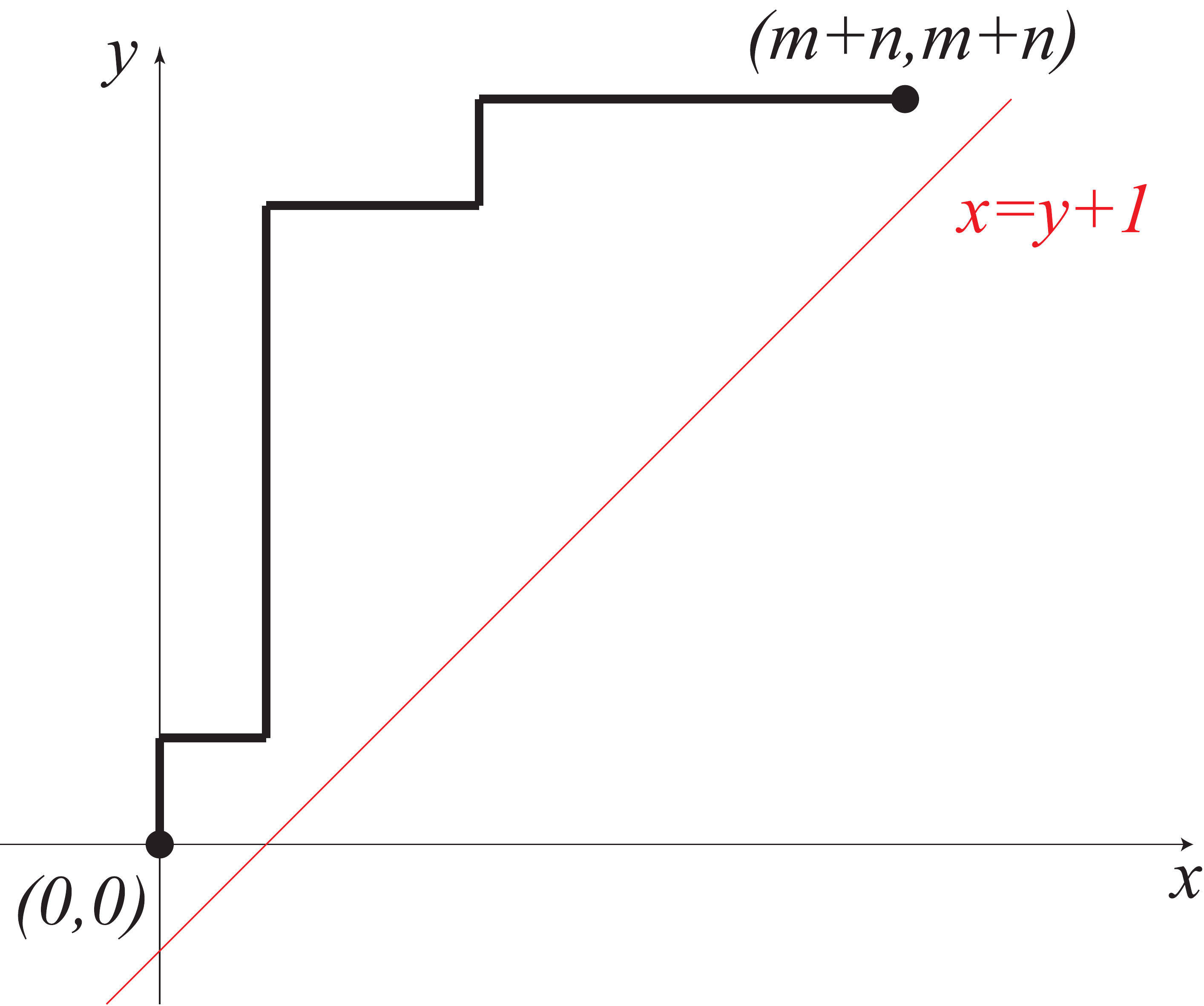}
\caption{Figure 4.2}
\label{fig 4.2}
\end{minipage}\qquad 
\begin{minipage}{.3\textwidth}
\centering
\includegraphics[width=\linewidth]{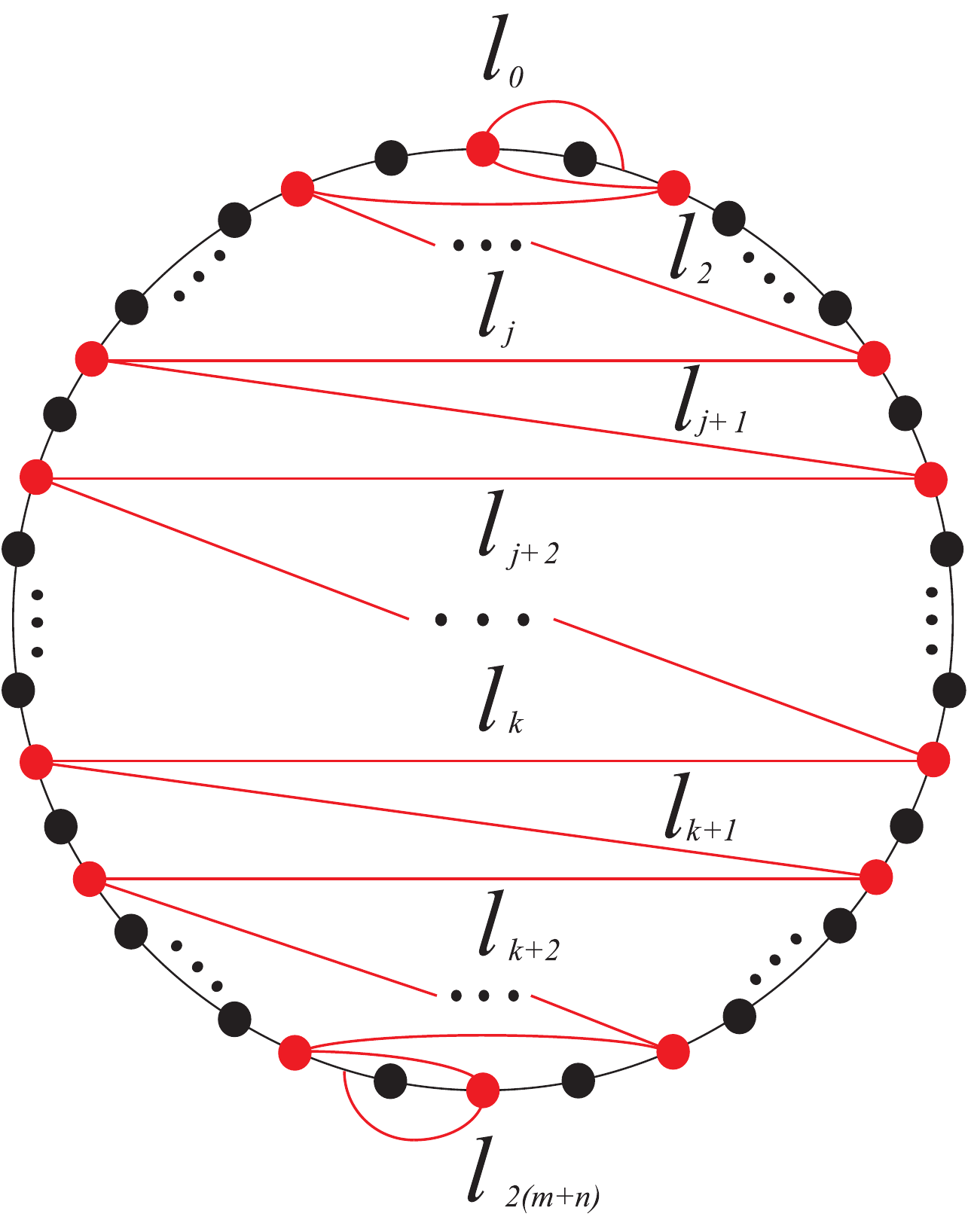}
\caption{Figure 4.3}
\label{fig 4.3}
\end{minipage}
\end{figure}

\begin{lemma}
\label{Lemma 4.4}For all $j=1,2,...,2\left( m+n\right) ,$ we have $%
r_{j}\left( C\right) =\pm 1.$
\end{lemma}

\begin{proof}
We observe that for any three successive lines $l_{2j},$ $l_{2j+1},$ and $%
l_{2j+2},$ where $j=1,$ $2,$ $...,$ $\left( m+n\right) -2$, we have $4$
possibilities (see \textrm{Figure} $4.4a)-d)$) for the intersections of $C$
with lines $l_{2j},$ $l_{2j+1},$ and $l_{2j+2}$ $:$

\begin{figure}[h]
\centering
\begin{minipage}{.22\textwidth}
\centering
\includegraphics[width=\linewidth]{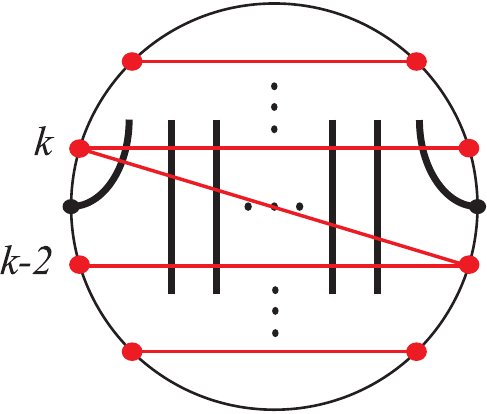}
\caption{Figure 4.4a)}
\label{fig 4.4a)}
\end{minipage}\qquad 
\begin{minipage}{.2\textwidth}
\centering
\includegraphics[width=\linewidth]{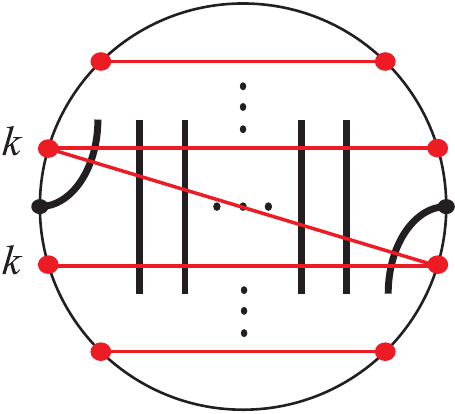}
\caption{Figure 4.4b)}
\label{fig 4.4b}
\end{minipage}\qquad \centering
\begin{minipage}{.2\textwidth}
\centering
\includegraphics[width=\linewidth]{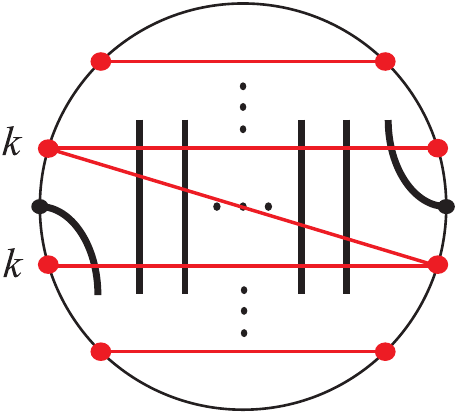}
\caption{Figure 4.4c)}
\label{fig 4.4c)}
\end{minipage}\qquad 
\begin{minipage}{.22\textwidth}
\centering
\includegraphics[width=\linewidth]{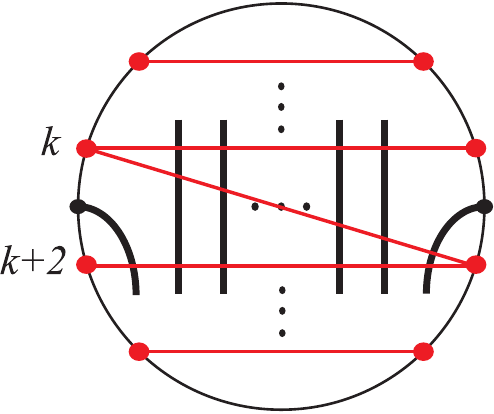}
\caption{Figure 4.4d)}
\label{fig 4.4d)}
\end{minipage}
\end{figure}

For instance, in the second case shown in \textrm{Figure} $4.4b)$, we have%
\begin{eqnarray*}
n_{2j}\left( C\right) &=&k,\text{ }n_{2j+1}\left( C\right) =k+1,\text{ and }%
n_{2j+2}\left( C\right) =k,\text{ thus} \\
r_{2j+1}\left( C\right) &=&1,\text{ and }r_{2j+2}\left( C\right) =-1.
\end{eqnarray*}%
Analogously, for the remaining cases shown in \textrm{Figure} $4.4a),$ 
\textrm{Figure} $4.4c),$ and \textrm{Figure} $4.4d).$ Therefore, for all $%
j=3,$ $4,$ $...,$ $2\left( m+n\right) -2,$ we have $r_{j}\left( C\right)
=\pm 1$.

For the remaining lines $l_{0},$ $l_{1},$ $l_{2},$ and $l_{2\left(
m+n-1\right) },$ $l_{2\left( m+n\right) -1},$ $l_{2\left( m+n\right) },$ we
see that%
\begin{eqnarray*}
n_{0}\left( C\right) &=&n_{2\left( m+n\right) }\left( C\right) =0,\text{ }%
n_{1}\left( C\right) =n_{2\left( m+n\right) -1}\left( C\right) =1,\text{ and}
\\
n_{2}\left( C\right) &=&0,\text{ }2;\text{ }n_{2\left( m+n-1\right) }\left(
C\right) =0,\text{ }2.
\end{eqnarray*}%
Therefore, we have $r_{1}\left( C\right) =1,$ $r_{2}\left( C\right) =-1$ or $%
r_{2}\left( C\right) =1$ and analogously $r_{2\left( m+n\right) -1}\left(
C\right) =1$ or $r_{2\left( m+n\right) -1}\left( C\right) =-1,$ $r_{2\left(
m+n\right) }\left( C\right) =-1$. Hence, we have $r_{j}\left( C\right) =\pm
1,$ for all $j=1,2,...,2\left(m+n\right) $ as we claimed.
\end{proof}

\smallskip

\noindent Now, let $C\in \mathfrak{Cat}_{m,n}$ and let%
\begin{equation*}
a_{j}\left( C\right) =\left\{ 
\begin{tabular}{lll}
$H$ & if & $r_{j}\left( C\right) =-1$ \\ 
$V$ & if & $r_{j}\left( C\right) =1$%
\end{tabular}%
\right. ,\text{ }j=1,2,...,2\left( m+n\right)
\end{equation*}%
We define a lattice path $\mathbf{p}\left( C\right) $ by putting%
\begin{equation*}
\mathbf{p}\left( C\right) =a_{1}\left( C\right) a_{2}\left( C\right)
...a_{2\left( m+n\right) }\left( C\right) .
\end{equation*}%
Clearly, for any $C\in \mathfrak{Cat}_{m,n},$ lattice path $\mathbf{p}\left(
C\right) $ of semilength $m+n$ from $\left( 0,0\right) $ to $\left( m+n,%
\text{ }m+n\right) $ never pass below the line $y=x,$ thus $\mathbf{p}\left(
C\right) $ is a Dyck path. Furthermore, we have the following result.

\begin{proposition}
\label{Proposition 4.5}The map $\psi :\mathfrak{Cat}_{m,n}\rightarrow 
\mathfrak{D}_{m,n}$ given by%
\begin{equation*}
\psi \left( C\right) =\mathbf{p}\left( C\right)
\end{equation*}%
is a bijection$.$
\end{proposition}

\begin{proof}
It suffices to show that $\psi :\mathfrak{Cat}_{m,n}\rightarrow \mathfrak{D}%
_{m,n}$ is injective. Let $C_{1},$ $C_{2}\in \mathfrak{Cat}_{m+n}$ and
suppose that%
\begin{equation*}
\psi \left( C_{1}\right) =\mathbf{p}\left( C_{1}\right) =\mathbf{p}\left(
C_{2}\right) =\psi \left( C_{2}\right) .
\end{equation*}%
therefore, for every $j=1,$ $2,$ $...,$ $2\left( m+n\right) ,$ $a_{j}\left(
C_{1}\right) =a_{j}\left( C_{2}\right) ,$ that is%
\begin{equation*}
r_{j}\left( C_{1}\right) =n_{j}\left( C_{1}\right) -n_{j-1}\left(
C_{1}\right) =n_{j}\left( C_{2}\right) -n_{j-1}\left( C_{2}\right)
=r_{j}\left( C_{2}\right)
\end{equation*}%
Since $n_{0}\left( C_{1}\right) =n_{0}\left( C_{2}\right) ,$ and $%
r_{1}\left( C_{1}\right) =r_{1}\left( C_{2}\right) ,$ we have $n_{1}\left(
C_{1}\right) =n_{1}\left( C_{2}\right) .$ By induction, we have%
\begin{equation*}
n_{j}\left( C_{1}\right) =n_{j}\left( C_{2}\right) ,\text{ }%
j=1,2,...,2\left( m+n\right) .
\end{equation*}%
It follows that $C_{1}$ and $C_{2}$ have identical minimal number of
intersections with all lines $l_{0},$ $l_{1},$ $...,$ $l_{2\left( m+n\right)
},$ hence $C_{1}=C_{2}$.
\end{proof}

\medskip

\noindent Following the notations used in \cite{NAR} (see pp. $9$ and Lemma $%
4A,$ pp. $12$), let $L\left( m,\text{ }n;\text{ }t\right) $ denote the set
of paths from $\left( 0,\text{ }0\right) $ to $\left( m,\text{ }n\right) $
not touching the line $y=x-t,$ where $t>0,$ and let $L\left( m,\text{ }n;%
\text{ }t,\text{ }s\right) $ denote the set of paths from $\left( 0,0\right) 
$ to $\left( m,n\right) $ not touching the lines%
\begin{equation*}
y=x-t\text{ and }y=x+s,
\end{equation*}%
where $t,$ $s>0.$ As it was shown in \cite{NAR}\footnote{%
Solution to the restricted ballot problem with two boundary conditions,
including very interesting historical comment of Kelvin and Maxwell's 
\textit{method of images,} is given in the classic book by W.Feller \cite%
{Fel}; see also \cite{Moh}, \cite{Taka}.}, we have%
\begin{eqnarray*}
\left\vert L\left( m,n;t\right) \right\vert &=&\binom{m+n}{n}-\binom{m+n}{m-t%
},\text{ and } \\
\left\vert L\left( m,n;t,s\right) \right\vert &=&\sum_{k\in \mathbb{Z}}\left[
\binom{m+n}{m-k\left( t+s\right) }_{+}-\binom{m+n}{m+k\left( t+s\right) +t}%
_{+}\right] ,
\end{eqnarray*}%
where%
\begin{equation*}
\binom{y}{z}_{+}=\left\{ 
\begin{tabular}{lll}
$\binom{y}{z}$ & $if$ & $y\geq z\geq 0$ \\ 
$0$ & $if$ & $y<0$ or $y<z$%
\end{tabular}%
\right .
\end{equation*}

\begin{theorem}
\label{Theorem 4.6}The number $T^{h}\left( m,n\right) $ of horizontally
excluded Catalan states $($see also \textrm{Figure} $4.6)$ is given by%
\begin{equation*}
T^{h}\left( m,n\right) =\sum_{i=0}^{\infty }\left[ \binom{2m+2n}{m-i\left(
n+3\right) -2}-2\binom{2m+2n}{m-i\left( n+3\right) -3}+\binom{2m+2n}{%
m-i\left( n+3\right) -4}\right]
\end{equation*}
\end{theorem}

\noindent Before we prove Theorem \ref{Theorem 4.6}, we note that if $i=2j,$
lines $l_{i},$ $l_{i+1}$ and $l_{i+2}$ may intersect a Catalan state $C$
satisfying horizontal line condition $H\left( m,n\right) $ in the pattern: $%
n_{i}\left( C\right) =n,$ $n_{i+1}\left( C\right) =n+1$ and $n_{i+2}\left(
C\right) =n$ (see \textrm{Figure} $4.5$)$.$

\begin{figure}[h]
\centering
\begin{minipage}{.27\textwidth}
\centering
\includegraphics[width=\linewidth]{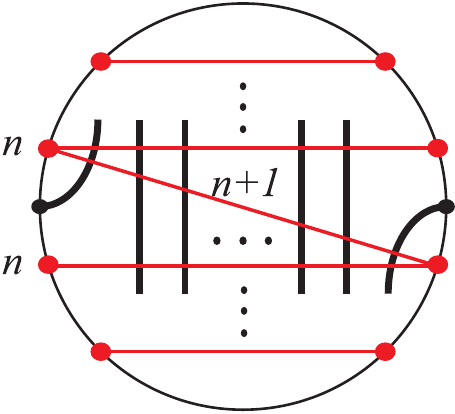}
\caption{Figure 4.5}
\label{fig 4.5}
\end{minipage}\qquad 
\begin{minipage}{.27\textwidth}
\centering
\includegraphics[width=\linewidth]{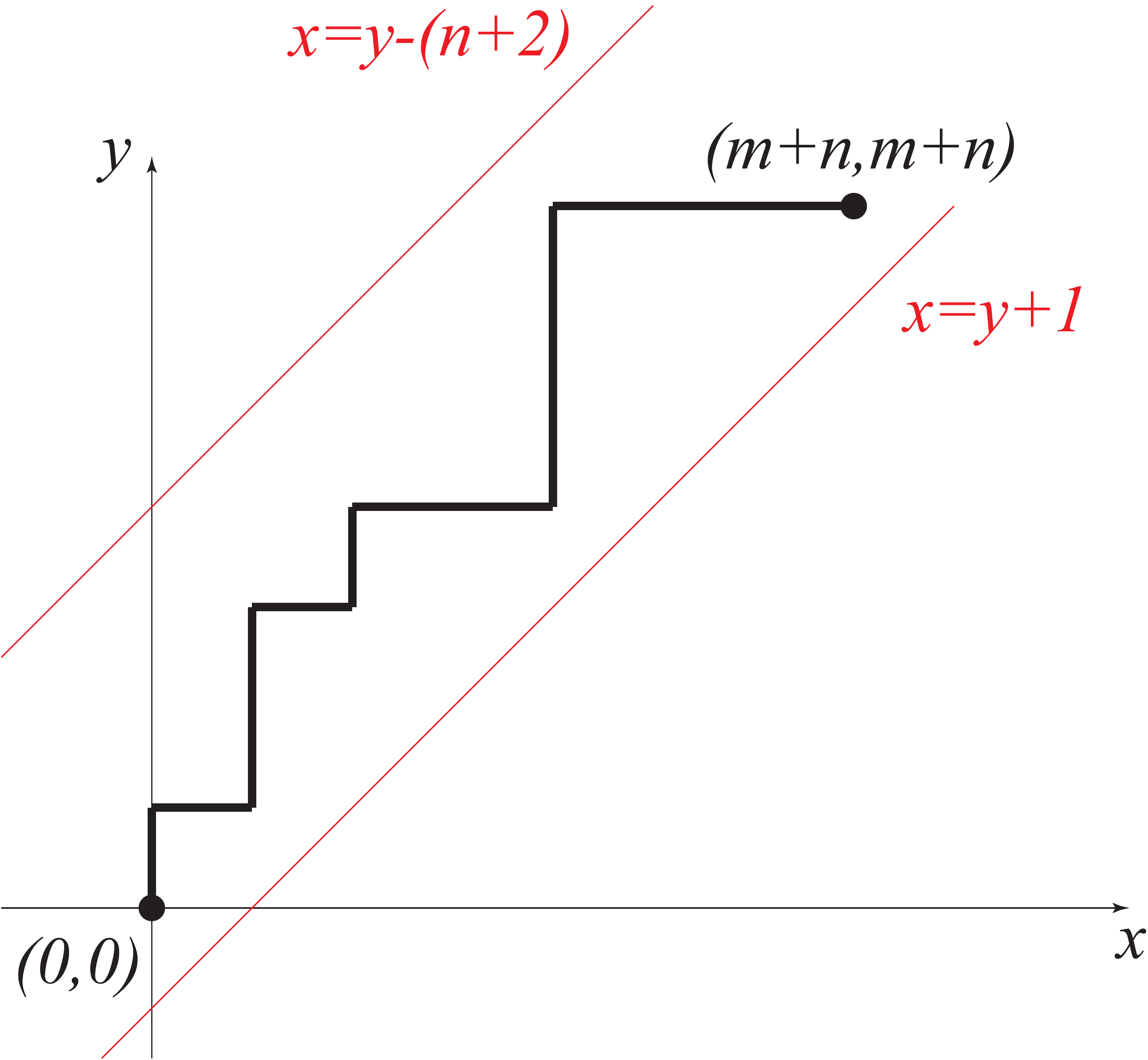}
\caption{Figure 4.6}
\label{fig 4.6}
\end{minipage}
\end{figure}

\noindent Recall, for a lattice path $\mathbf{p}=a_{1}a_{2}...a_{2k}$ that
starts from $\left( 0,0\right) $, one defines%
\begin{eqnarray*}
H_{j}\left( \mathbf{p}\right) &=&\#\left\{ a_{i}\text{ }|\text{ }a_{i}=H,%
\text{ }i=1,\text{ }2,\text{ }...,\text{ }j\right\} \text{ - number of
horizontal steps} \\
V_{j}\left( \mathbf{p}\right) &=&\#\left\{ a_{i}\text{ }|\text{ }a_{i}=V,%
\text{ }i=1,\text{ }2,\text{ }...,\text{ }j\right\} \text{ - number of
vertical steps}
\end{eqnarray*}%
where $1\leq j\leq k,$ and 
\begin{equation*}
S_{j}\left( \mathbf{p}\right) =V_{j}\left( \mathbf{p}\right) -H_{j}\left( 
\mathbf{p}\right) ,\text{ }1\leq j\leq k.
\end{equation*}%
Obviously, path $\mathbf{p}\left( C\right) ,$ for any $j=1,$ $2,$ $...,$ $%
2\left( m+n\right) ,$ we have$:$%
\begin{equation*}
S_{j}\left( \mathbf{p}\left( C\right) \right) =V_{j}\left( \mathbf{p}\right)
-H_{j}\left( \mathbf{p}\right) =\sum_{i=1}^{j}r_{i}\left( C\right)
=n_{j}\left( C\right) \geq 0.
\end{equation*}

\begin{proof}
Using bijection $\psi :\mathfrak{Cat}_{m,n}\rightarrow \mathfrak{D}_{m,n}$
from Proposition \ref{Proposition 4.5}, we have:%
\begin{eqnarray*}
T^{h}\left( m,n\right) &=&\left\vert \mathbf{T}^{h}\left( m,n\right)
\right\vert =\left\vert \left\{ C\in \mathfrak{Cat}_{m,n}\text{ }|\text{ }%
S_{j}\left( \mathbf{p}\left( C\right) \right) >n+1,\text{ for some }0\leq
j\leq 2\left( m+n\right) \right\} \right\vert \\
&=&\left\vert \mathfrak{D}_{m,n}\right\vert -\left\vert \left\{ \mathbf{p}%
\in \mathfrak{D}_{m,n}\text{ }|\text{ }1\leq S_{j}\left( \mathbf{p}\right)
\leq n+1,\text{ for all }0\leq j\leq 2\left( m+n\right) \right\} \right\vert
\end{eqnarray*}%
Since

\begin{equation*}
\left\vert \mathfrak{D}_{m,n}\right\vert =\left\vert L\left( m+n,\text{ }m+n;%
\text{ }1\right) \right\vert
\end{equation*}%
and

\begin{equation*}
\left\vert \left\{ \mathbf{p}\in \mathfrak{D}_{m,n}\text{ }|\text{ }1\leq
S_{j}\left( \mathbf{p}\right) \leq n+1,\text{ for all }0\leq j\leq 2\left(
m+n\right) \right\} \right\vert =\left\vert L\left( m+n,m+n;1,n+2\right)
\right\vert
\end{equation*}%
After standard algebraic simplifications, we have

\begin{eqnarray*}
T^{h}\left( m,n\right) &=&\left\vert L\left( m+n,m+n;\text{ }1\right)
\right\vert -\left\vert L\left( m+n,m+n;\text{ }1,\text{ }n+2\right)
\right\vert \\
&=&\sum_{i=0}^{\infty }\left[ \binom{2m+2n}{m-i\left( n+3\right) -2}-2\binom{%
2m+2n}{m-i\left( n+3\right) -3}+\binom{2m+2n}{m-i\left( n+3\right) -4}\right]%
.
\end{eqnarray*}%
This finishes our proof.
\end{proof}

\medskip

\begin{corollary}
\label{Corollary 4.7}The number of all realizable Catalan states by $m\times
n$ lattice of crossings $T_{\left( m,n\right) }$ is given by%
\begin{equation*}
\left\vert \mathfrak{Cat}_{m,n}^{R}\right\vert =\frac{1}{m+n+1}\binom{%
2\left( m+n\right) }{m+n}-T^{h}\left( m,n\right) -T^{h}\left( n,m\right) .
\end{equation*}%
In particular, for $m=n,$ 
\begin{equation*}
\left\vert \mathfrak{Cat}_{n,n}^{R}\right\vert =\frac{1}{2n+1}\binom{4n}{2n}%
-2\left( \binom{4n}{n-2}-2\binom{4n}{n-3}+\binom{4n}{n-4}\right)
\end{equation*}
\end{corollary}

\section{Acknowledgments}

J.~H.~Przytycki was partially supported by the NSA-AMS 091111 grant, and by
the GWU REF grant. Authors would like to thank Ivan Dynnikov and Krzysztof
Putyra for many useful computations/discussions.

\begin{equation*}
\begin{tabular}{ll}
Mieczyslaw K. Dabkowski & Changsong Li \\ 
Department of Mathematical Sciences & Department of Mathematical Sciences \\ 
University of Texas at Dallas & University of Texas at Dallas \\ 
Richardson TX, 75080 & Richardson TX, 75080 \\ 
\emph{e-mail}: \texttt{mdab@utdallas.edu} & \emph{e-mail}: \texttt{%
changsong.li@utdallas.edu} \\ 
& \multicolumn{1}{c}{} \\ 
Jozef H. Przytycki & \multicolumn{1}{c}{} \\ 
Department of Mathematics &  \\ 
The George Washington University &  \\ 
Washington, DC 20052 &  \\ 
\emph{e-mail}: \texttt{przytyck@gwu.edu} &  \\ 
University of Maryland College Park, &  \\ 
and University of Gda\'{n}sk & 
\end{tabular}%
\end{equation*}

\end{document}